\documentclass [11pt]{article}

\usepackage{hyperref}
\usepackage{amsmath, amssymb, amsfonts, amsthm, enumerate}
\usepackage{tikz}
\usepackage{verbatim}
\date{}

\newtheorem{theorem}{\bf Theorem}[section]
\newtheorem{claim}[theorem]{\bf Claim}
\newtheorem{lemma}[theorem]{\bf Lemma}

\newtheorem{proposition}[theorem]{\bf Proposition}

\newtheorem{observation}[theorem]{\bf Observation}

\newcommand{\rt}{\right}
\newcommand{\lt}{\left}

\newcommand{\cG}{\mathcal{G}}

\let\eps=\varepsilon
\let\theta=\vartheta
\let\rho=\varrho

\title{\vspace{-1cm} The number of Hamiltonian decompositions of regular graphs.}

\author{
Roman Glebov
\thanks{School of Computer Science and Engineering, Hebrew University, Jerusalem 9190401, Israel and
Department of Mathematics, ETH, 8092 Zurich, Switzerland.
 {\tt roman.l.glebov@gmail.com}.
Research supported by the ERC grant ``High-dimensional combinatorics'' at the Hebrew University.}
\and
Zur Luria
\thanks{Institute of Theoretical Studies, ETH, 8092 Zurich, Switzerland. {\tt zluria@gmail.com}.
Research supported by Dr.~Max R\"ossler, the Walter Haefner Foundation and the ETH Foundation.}
\and
Benny Sudakov
\thanks{Department of Mathematics, ETH, 8092 Zurich, Switzerland. {\tt benjamin.sudakov@math.ethz.ch}. Research supported in part by SNSF grant 200021-149111.}
}

\begin{document}
\maketitle

\begin{abstract}
A Hamilton cycle in a graph $\Gamma$ is a cycle passing through every vertex of $\Gamma$.
A Hamiltonian decomposition of $\Gamma$ is a partition of its edge set into disjoint Hamilton cycles.
One of the oldest results in graph theory is Walecki's theorem from the 19th century,
showing that a complete graph $K_n$ on an odd number of vertices $n$ has a Hamiltonian decomposition. This result
was recently greatly extended by K\"{u}hn and Osthus. They proved that every
$r$-regular $n$-vertex graph $\Gamma$ with even degree $r=cn$ for some fixed $c>1/2$ has a Hamiltonian decomposition, provided $n=n(c)$ is sufficiently
large. In this paper we address the natural question of estimating $H(\Gamma)$, the number of such decompositions of $\Gamma$. Our main result is that $H(\Gamma)=r^{(1+o(1))nr/2}$. In particular, the number of Hamiltonian decompositions of $K_n$ is $n^{(1+o(1))n^2/2}$.
\end{abstract}

\maketitle

\section{Introduction}
\label{sec:intro}
A \emph{Hamilton cycle} in a graph $\Gamma$ is a cycle passing through each vertex of
$\Gamma$, and a graph is \emph{Hamiltonian} if it contains
a Hamilton cycle. Hamiltonicity, named after Sir Rowan Hamilton who studied it in the 1850s, is one of the most important
and extensively studied concepts in graph theory. It is well known that deciding Hamiltonicity is an NP-complete
problem, and thus one does not expect a simple sufficient condition for Hamiltonicity for general graphs.
Once Hamiltonicity is established, it is very natural to strengthen such a result by showing that the graph in question has many
edge-disjoint Hamilton cycles, or even has a \textit{Hamiltonian decomposition}, which is a partition of the edge set of the graph into disjoint Hamilton cycles. Clearly a Hamiltonian decomposition is only possible when $\Gamma$ is $r$-regular for some even $r$.
In 1890, Walecki gave a celebrated construction of a Hamiltonian decomposition of the complete graph $K_n$ for every odd $n$.
For a description of his construction, see, e.g.,~\cite{Al09}.

The work of Walecki was extended by various researchers, who proved that more general families of graphs admit a Hamiltonian decomposition, see, e.g., \cite{Rin, BFM89, KW, KuOs14,CLKOT15+} and a survey \cite{AlBe90}, which gives an overview of many results on this topic. One of the  classical results on Hamiltonicity of graphs is a theorem of Dirac which says that every $n$-vertex graph with minimum degree $n/2$ has a Hamilton cycle.
Recently K\"{u}hn and Osthus \cite{KuOs14} obtained a far reaching generalization of both Dirac's result for regular graphs and Walecki's decomposition. They proved that every
$r$-regular $n$-vertex graph $\Gamma$ with even degree $r=cn$ for some fixed $c>1/2$ has a Hamiltonian decomposition, provided $n=n(c)$ is sufficiently large.

The question of estimating the number of Hamilton cycles in various classes of graphs
has also been intensively studied, see, e.g.,~\cite{Ja94, F, CK, Kr12, GK13, FKS12} and the references therein.
However, to the best of our knowledge, so far only few results have been established on the number of Hamiltonian decompositions,
see, e.g.,~\cite{FoSkZa09} and the references therein.
In this paper we study the first natural question of this sort, providing counting versions of Walecki's and K\"{u}hn-Osthus' results. We are interested in $H(\Gamma)$, the number of Hamiltonian decompositions of a given $r$-regular $n$-vertex graph $\Gamma$ with even degree $r=cn$ for some fixed constant $c>1/2$.

To upper bound the number of Hamilton cycles in a regular graph, one can use
the standard upper bound for the permanent of a 0-1 matrix proved by Br\'{e}gman~\cite{Br73}
(solving the famous Minc conjecture).  This permanent-based approach to
Hamiltonicity problems was used for the first time by Alon~\cite{Alon} to bound the number of Hamilton paths in
tournaments (see also~\cite{FK, Kr12, GK13, KKO, FKS12} for additional applications).
Let $S_n$ be the set of all permutations of the set $[n]$. The \emph{permanent} of an $n\times n$ matrix $A$ is
defined as $per(A)=\sum_{\sigma\in S_n} \prod_{i=1}^n A_{i\sigma(i)}$.
Note that when $A=A_\Gamma$ is the 0-1 adjacency matrix of a graph $\Gamma$,
a summand in the permanent is $1$ if it corresponds to a spanning subgraph of $\Gamma$
whose connected components are either cycles or isolated edges, and the summand is $0$ otherwise.
Therefore, the number of Hamilton cycles in $\Gamma$ is bounded from above by $per(A_\Gamma)$. Combining this observation with
Br\'{e}gman's~\cite{Br73} upper bound, we see that an $r$-regular graph has at most
$(r!)^{n/r}$ Hamilton cycles. By choosing any one of these Hamilton cycles and  deleting it, we obtain an $(r-2)$-regular graph. The number of Hamilton cycles in the new graph can again be bounded from above using Br\'{e}gman's theorem.
Continuing this process and multiplying all the estimates for regularities $r, r-2, r-4, \ldots $
we can use Stirling's formula to deduce the following upper bound.

\begin{proposition}
\label{thm:upper_bound}
For every $r=r(n)\rightarrow \infty$, the number of Hamiltonian decompositions of an $r$-regular graph $\Gamma$ of order $n$ is at most
\[
\left((1+o(1))\frac{r}{e^2}\right)^{nr/2}.
\]
\end{proposition}

Our main result is the corresponding lower bound, which together with the last proposition determines
asymptotically the number of Hamiltonian decompositions of dense regular graphs.

\begin{theorem}
\label{thm:counting_walecki}
Let $c>1/2$ be a constant and let $\Gamma$ be an $n$-vertex $r$-regular graph with even degree $r \geq cn$.
Then the number of Hamiltonian decompositions of $\Gamma$ satisfies
\[H(\Gamma)=r^{(1+o(1))rn/2} \,. \]
\end{theorem}

\noindent
In particular the number of Hamiltonian decompositions of $K_n$ is $n^{(1+o(1))n^2/2}$.

\subsection{Outline of the Proof}
\label{sec:outline}
Let $\Gamma$ be an $n$-vertex $r$-regular graph with even degree $r \geq cn$ and $c>1/2$. Let $\delta$ be an arbitrary constant $\delta<\frac{1}{2}(c-1/2)$ and let $A$ and $B$ be two subsets of $V(\Gamma)$ such that $|A| \geq \delta^2 n$ and $|B| \geq (1/2-\delta)n$. Every vertex in $A$ has at least $cn-(n-|B|)>\delta n$ neighbors in $B$. Therefore the number of edges between $A$ and $B$ satisfies $e_\Gamma(A,B) \geq \delta n|A|/2\geq \delta^3n^2/2$. (If $A$ and $B$ have a nonempty intersection then $e_\Gamma(A,B)$ counts the edges inside $A \cap B$ with multiplicity one.) Therefore, our main result follows from the following more general statement (by taking $\gamma=\delta^3/2$ and $\eps \rightarrow 0$).

\begin{theorem}
\label{counting}
For every $c>0$ and $0<\eps<\frac{1}{10}$, there exists $\delta>0$ such that for any constant $\gamma>0$ and sufficiently large $n$ the following holds.
Let $\Gamma$ be an $n$-vertex $r$-regular graph with even degree $r \geq c n$ such that
there are at least $\gamma n^2$ edges between any two subsets $A,B \subseteq V(G)$ satisfying  $|A| \geq \delta^2 n$ and $|B| \geq (1/2-\delta)n$.
Then the number of Hamiltonian decompositions of $\Gamma$ is at least $r^{(1-5\eps)rn/2}$.
\end{theorem}

The proof of this statement follows the same general strategy as the proof of Proposition~\ref{thm:upper_bound}.
Namely, we remove Hamilton cycles one by one, giving a lower bound on the number of possibilities at each step.
The two additional ingredients we need are a way to construct many Hamilton cycles in a regular graph,
and sufficient conditions for the existence of a Hamiltonian decomposition in a graph.

The proof consists of three parts.
In the first part, we split $\Gamma$ into three edge-disjoint spanning graphs $G$, $F$, and $R$.
The graph $G$ contains most of the edges of $\Gamma$,
$F$ satisfies a pseudo-randomness condition similar to the expander mixing lemma,
and $R$ is a robust expander, a concept that is defined below.

In the second part, we use some ideas of Ferber, Krivelevich, and Sudakov~\cite{FKS12}
to repeatedly find Hamilton cycles in $G \cup F$.
Their approach, based on the permanent of the adjacency matrix,
gives us at every step many $2$-regular subgraphs of the current graph which are unions of a small number of cycles.
Using edges of $F$ we can transform each of these subgraphs into a Hamilton cycle.
This part of the proof gives the asymptotic count in the theorem.
We stop at some point when all but a small constant fraction of the edges of $G \cup F$ have been covered.

In the third part of the proof, we need to show that the remaining edges of $\Gamma$ can be partitioned into Hamilton cycles.
Here, we use a recent result of K\"{u}hn and Osthus~\cite{KuOs14}.
Given parameters $0<\nu\leq \tau<1$, a graph $\cG$ and a set of vertices $S \subseteq V(\cG)$,
the robust $\nu$-neighborhood of $S$ is the set of vertices that have at least $\nu n$ neighbors in $S$, and is denoted by $RN_{\nu,\cG}(S)$.
A graph $\cG$ is a robust $(\nu,\tau)$-expander if
\[
RN_{\nu,\cG}(S) \geq |S|+\nu n \text{ for all $S$ such that $ \tau n\leq |S|\leq (1-\tau) n $}.
\]
Note in particular that this is a monotone property under addition of edges.
The following powerful theorem of K\"{u}hn and Osthus~\cite{KuOs14},
which they derived from the corresponding digraph result~\cite{KuOs13},
provides sufficient conditions for a graph to admit a Hamiltonian decomposition.
\begin{theorem}[Theorem 1.2 in \cite{KuOs14}]
\label{thm:suff_conditions}
For every $\alpha>0$ there exists $\tau>0$ such that for every $\nu>0$ there exists $n_0(\alpha,\nu,\tau)$ for which the following holds.
Suppose that
\begin{itemize}
\item $\cG$ is an $r$-regular graph on $n\geq n_0$ vertices, where $r \geq \alpha n$ is even, and
\item $\cG$ is a robust $(\nu,\tau)$-expander.
\end{itemize}
Then $\cG$ has a Hamiltonian decomposition.
\end{theorem}

Since $R$ is a robust expander, the graph consisting of all remaining edges of $\Gamma$ that are not covered by the Hamilton cycles found in the second part
satisfies the conditions of Theorem~\ref{thm:suff_conditions},
and therefore we can complete our packing of edge-disjoint Hamilton cycles to a Hamiltonian decomposition of $\Gamma$.

Throughout the proof, we sometimes identify a graph with its edge set.
For the sake of clarity of presentation, we omit floor and ceiling signs.

\section{Proof of Theorem~\ref{counting}}
\label{sec:proof}
Given $c, \gamma>0$ and $0<\eps<\frac{1}{10}$, choose $\alpha = 3 \eps c$ and let $\tau = \tau(\alpha)$ be as in Theorem~\ref{thm:suff_conditions}.
Next, choose $\delta \leq \min(\eps c/5,\tau/2)$ and set $\nu=\min (\delta, \eps \gamma/2)$. Finally,
let $n_0 = n_0(\alpha, \nu,\tau)$ be as in Theorem~\ref{thm:suff_conditions}.
Throughout this section we will always assume that $n>n_0$ and is sufficiently large whenever it is necessary.

We first show how to find three edge-disjoint graphs $G$, $F$, and $R$ with the desired properties.
\begin{lemma}
\label{lem:main}
The edges of $\Gamma$ can be partitioned into three subgraphs $G$, $F$, and $R$ such that:
\begin{itemize}
\item
The graph $G$ is $d$-regular with even degree $d\geq (1-2 \eps)r$.
\item
For any two sets $A,B\subseteq [n]$ of sizes at least $|A|\geq\delta^2 n$ and $|B|\geq(1/2-\delta)n$, there exist at least $n^{1.6}$ edges in $F$ between $A$ and $B$.
\item
The graph $R$ is a robust $(\nu,\tau)$-expander.
\end{itemize}
\end{lemma}

\begin{proof}
Consider the edge partition of $\Gamma$ into three edge-disjoint random subgraphs $F$, $R^*$, and $G^*$,
where every edge gets into $F$ with probability $1/\log n$, into $R^*$ with probability $\eps$,
and into $G^*$ with probability $1-\eps-1/\log n$.

First, we need to find a dense regular spanning subgraph $G\subseteq G^*$. To do so we will use the following lemma, which we prove in the Appendix to this paper.
\begin{lemma}\label{lem:regular_subgraph}
Let $\eps_0, \gamma_0>0$ be constants, $c_0\geq \eps_0$, and let $n$ be sufficiently large.
Let $G^*$ be an $n$-vertex graph whose vertex degrees are all between $c_0 n - n^{2/3}$ and $c_0 n +n^{2/3}$, such that there are at least $\gamma_0 n^2$ edges between any two subsets $A,B \subseteq V(G^*)$ satisfying $|A|\geq c_0 n /3$ and $|B|\geq n/2$.
Then $G^*$ has a regular spanning subgraph of degree $2d$, where $d = \lt\lceil \frac{(c_0-\eps_0)n}{2}\rt\rceil$.
\end{lemma}

We claim that with high probability $G^*$ satisfies the conditions of this Lemma, with the parameters $c_0 = (1 - \eps - 1/\log(n))\frac{r}{n}$ and $\gamma_0 = \frac{\gamma}{2}$.
Indeed, this follows from a straightforward application of Chernoff's inequality and the union bound
after observing that $c_0/3 > \delta^2$.
We apply Lemma~\ref{lem:regular_subgraph} with $\eps_0 = \frac{\eps r}{2n}<c_0$, obtaining the desired $d$-regular graph $G$
with $d\geq (1-2 \eps)r$.

Next, consider the graph $F$.
Between any two sets $A,B\subseteq [n]$ of sizes $|A|\geq\delta^2 n$ and $|B|\geq(1/2-\delta)n$,
there are at least $\gamma n^2 $ edges of $\Gamma$.
Therefore, the number of edges between $A$ and $B$ in $F$
is distributed binomially with parameters $|E_\Gamma(A,B)| \geq \gamma n^2$ and $\log^{-1} n$.
Using Chernoff-type estimates again, we can see that the probability of having less than $n^{1.6}$ edges in $F$ between $A$ and $B$ is $\ll 2^{-2n}$.
Taking a union bound over all pairs of subsets shows that
with high probability $\Gamma$ has the desired property.

Note that the same argument as above using Chernoff-type estimates and the union bound implies that
with high probability there are at least $\eps \gamma n^2/2 $ edges of $R^*$ between any two sets $A,B\subseteq [n]$ of sizes $|A|\geq\delta^2 n$ and $|B|\geq(1/2-\delta)n$. Using this we prove that $R^*$ is robust expander.
First fix a subset $S$ of the vertices of $R^*$ such that $\delta^2 n<\tau n \leq |S| < n/2$, and consider the set $B$ of vertices of
$R^*$ which have less than $\nu n\leq \eps \gamma n/2$ neighbors in $S$.
If there are at least $(1/2-\delta)n$ such vertices, then between $B$
and $S$ there are less than $\eps \gamma n^2/2$ edges, contradiction. Therefore the
robust $ \nu $-neighborhood of $S$ has size at least $(1/2+\delta)n \geq |S|+\nu n$. Now consider a subset
$S$ of size  $ n/2 \leq |S| \leq (1-\tau)n\leq n- 2\delta n$ and let $A$ be the set of vertices of $R^*$
which have less than $\nu n\leq \eps \gamma n/2$ neighbors in $S$.
If there are at least $\delta^2 n$ such vertices then there are less than $\eps \gamma n^2/2$ edges between $A$
and $S$, again obtaining a contradiction. Therefore the robust
$ \nu $-neighborhood of $S$ has size at least $n -\delta^2 n \geq |S|+\nu n$ and so
$R^*$ is a robust $(\nu,\tau)$-expander. Finally, define $R$ to be the union of $R^*$ together with all the edges from  $G^* \setminus G$.
By monotonicity, $R$ is a robust expander as well.
\end{proof}

As we already mentioned above, our proof proceeds by successively removing Hamilton cycles from $G \cup F$.
The following lemma provides a lower bound on the number of choices in each such step.

\begin{lemma}
\label{lem:2-factors}
Let $\eps>0$ and let $G$ and $F$ be edge-disjoint graphs on the vertex set $[n]$
such  that the graph $G$ is $d$-regular with $d\geq \eps r=\eps cn$
and for any two sets $A,B\subseteq [n]$ of sizes at least $|A|\geq\delta^2 n$ and $|B|\geq(1/2-\delta)n$,
there are at least $d n^{0.6}$ edges of $F$ between $A$ and $B$.
Then there exist at least $(1+o(1))^n(d/e)^n$ Hamilton cycles $H$ in $G\cup F$ with the following properties.
For each cycle $H$
there are edge sets $E_G \subseteq G$ and $E_F \subseteq F$ such that
\begin{itemize}
\item
$G':=(G\cup E_F)\setminus (H\cup E_G)$ is $(d-2)$-regular.
\item
$F':=F\setminus (H \cup E_F)$ has the property that
for any $A,B\subseteq [n]$ of sizes at least $|A|\geq\delta^2 n$ and $|B|\geq(1/2-\delta)n$, there are at least $(d-2) n^{0.6}$ edges of $F'$ between $A$ and $B$.
\end{itemize}
\end{lemma}

\noindent
{\bf Remark.}\, In order to keep the graph obtained from $G$ after removing the cycle $H$ regular we need to add and delete some additional edges.
This is achieved by using the sets of edges $E_G$ and $E_F$. Note that the resulting graphs $G',F'$, and $H$ are edge-disjoint.

\vspace{0.1cm}
\begin{proof}
We call a spanning subgraph of $G\cup F$ a {\em partial HC} (for Hamilton cycle) if one of its components is a path,
and all other components are cycles or isolated edges.
We also call a spanning collection of vertex-disjoint cycles and isolated edges a {\em $(\leq 2)$-factor}.
The following claim allows us to reduce the number of connected components in a $2$-factor
by concatenating two of its cycles into a path.

\begin{claim}
\label{cla:extension}
Let $H$ be a $2$-factor in $G\cup F$ with more than one cycle. There exist edges $e\in (G\cup F)\setminus H$ and $e_1', e'_2\in H$
such that $H\cup\{e\}\setminus \{e'_1, e'_2\}$ is a partial HC with one less component (and one less edge) than $H$.
\end{claim}

\begin{proof}
Consider an arbitrary cycle $C$ of $H$ on at most $n/2$ vertices.
If $|C|<d$ then there exists an edge $e\in E(G)$ connecting $C $ with a different component of $H$.
If $|C|\geq d\geq \eps cn \geq \delta^2 n$, then again there exists an edge $e\in E(F)$ connecting $C$ with a different component of $H$.
In both cases, obviously $e\not\in H$.
Now any two edges $e_1', e'_2\in E(H)$ adjacent to different endpoints of $e$ satisfy the requirements of the claim.
\end{proof}

The following technical claim allows us to either extend a path in a partial HC,
thereby reducing the number of connected components, or to close this path into a cycle.

\begin{claim}
\label{cla:rotator}
Let $H$ be a a partial HC in $G\cup F$. Then one of the following holds:
\begin{itemize}
\item
There exist edges $e,e_1,e_2\in (G\cup F)\setminus H$ and $e',e'_1, e'_2\in H$
such that $H\cup\{e\}\setminus \{e'\}$, or $H\cup\{e\}$, or $H\cup\{e,e_1,e_2\}\setminus \{e',e'_1, e'_2\}$,
or $H\cup\{e,e_1,e_2\}\setminus \{e'_1, e'_2\}$
is a partial HC with one less component than $H$ (and at least as many edges as $H$).
\item
There exist edges $e_1, e_2, e_3\in (G\cup F)\setminus H$, $e'_1, e'_2, e'_3\in H$, and $f\in F\setminus H$
such that $H\cup\{e_1, e_2,e_3,f\}\setminus \{e'_1, e'_2, e'_3\}$ is a $(\leq 2)$-factor in $G\cup F$
with the same number of components as $H$ (and one more edge than $H$).
\end{itemize}
\end{claim}

\begin{proof}

At least one component of $H$ is a path, choose $P=v_1, \ldots, v_t$ be a longest such. If $v_1$ or $v_t$ have one of their neighbors outside $P$,
say, $(v_1z)\in E(G)$ with $z\not \in P$,
then $(v_1z)$ connects $P$ to a component $D$ of $H$ with $z\in D$.
If $D$ is an isolated edge, then $H\cup (v_1z)$ satisfies the conditions of the first bullet.
Otherwise, $D$ is a cycle and we can choose an arbitrary edge $e'\in E(D)$ incident to $z$,
and $H\cup(v_1z)\setminus e'$ satisfies the conditions of the first bullet.

Therefore we can assume that all the neighbors of $v_1, v_t$ in $G$ lie in $P$.
Then we will use a variant of the celebrated rotation-extension technique of P\'osa~\cite{Posa}, which was
developed by the third author together with Vu (see the proof of Theorem 4.1 in \cite{SV}).

We split $P$ into $s=2\delta^{-1}$ segments $I_1, \ldots, I_s$ of size $|P|/s \leq \delta n/2$ each, and consider the graph $G$.
By the pigeonhole principle, there exist $p,q\in [s]$
such that $v_1$ has at least $d/s \geq \eps c \delta n/2$ neighbors in $I_p$,
and $v_t$ has at least $\eps c \delta n/2$ neighbors in $I_q$.
If $p=q$, split $I_p$ further into two segments
$J_1$ and $J_2$
such that the {\em interior} of $J_1$,
i.e., all its vertices except of its endpoints,
contain at least $\eps c \delta n/5$ neighbors of $v_1$,
and $v_t$ has at least $\eps c \delta n/5$ neighbors in the interior of $J_2$.
If $p\neq q$, simply set $J_1=I_p$ and $J_2=I_q$.
For every neighbor $v_i$ of $v_t$ from the interior of $J_2$,
we can {\em rotate} the path $P$ keeping its endpoint $v_1$ fixed and using $v_i$ as the {\em pivot}. This means that we add the edge $(v_i,v_t)$,
delete the edge $(v_i,v_{i+1})$, and consider the new path $v_1, \ldots, v_i, v_t, \ldots, v_{i+1}$
with endpoints $v_1$ and $v_{i+1}$.
Setting $A$ to be the set of successors of the neighbors of $v_t$ in the interior of $J_2$,
then for every $a\in A$, we obtain a $(v_1,a)$-path on the vertex set $V(P)$ differing from $P$ in exactly one edge from $J_2$.
Similarly, for every neighbor $v_j$ of $v_1$ in the interior of $J_1$,
every $(v_1,a)$-path can be further rotated using the pivot $v_j$. Note that since $J_1$ and $J_2$ are disjoint, all these paths we constructed
using pivots in $J_1$ traverse $J_2$ in the same direction. Therefore the new endpoints we will obtain do not depend on $a$.
Thus we constructed a set $B$ such that for every $a \in A, b \in B$,
there exists a $(b,a)$-path in $G$ on the vertex set $V(P)$
differing from $P$ in exactly two edges from $J_1$ and $J_2$.
Furthermore, $|A|,|B|\geq \eps c \delta n/5\geq \delta^2 n$,
and all these paths traverse all other segments of $P$ outside of $J_1$ and $J_2$
in the same direction.

If a vertex $x\in A\cup B$ has a neighbor outside $V(P)$ in $F\cup G$,
we extend $P$ similarly to the case discussed above,
obtaining a partial HC with one less component than $H$ (and at least as many edges as $H$),
as desired in the first bullet of the claim.
Therefore we can assume that all the neighbors of vertices in $A \cup B$ in $G$ lie in $P$.
Then for every $a\in A, b\in B$ and every neighbor $v_\ell$ of $a$ in $G$ outside of $J_1$ and $J_2$,
we can rotate every $(b,a)$-path further using the pivot $v_\ell$ to obtain a new endpoint $c$ and a $(b,c)$-path.
Note that, since we did not touch any edge outside $J_1$ and $J_2$ before and the parts of $P$ outside these intervals are traversed by all paths in the same direction,
the vertex $c$ is determined only by the neighbor $v_\ell$ of $a$ and does not depend on the other endpoint $b$.
Therefore, we obtain a set $S$ such that for every $b\in B$ and $c \in S$,
there exists a $(b,c)$-path in $G$ on the same vertex set as $P$, which  differs from $P$ in exactly three edges.
Furthermore, we know that $F$ has an edge between any two subsets of size $\delta^2n$ and $(1/2-\delta)n$.
This implies that  $|N_F(A)| \geq n-(1/2-\delta)n=(1/2+\delta)n$. Therefore
\[|S|\geq |N_F(A)|-|J_1\cup J_2|\geq (1/2+\delta)n- \delta n \geq n/2.\]
Similarly we also have that $|N_F(B)| \geq (1/2+\delta)n$. Thus there exists an edge $f=(b,c)\in E(F)$ between vertices $b\in B$ and $c\in S$.
This edge closes the $(b,c)$-path into a cycle, see Figure 1. Hence we can use the edge $f$, together with the three edges from the rotations of $P$ which produced the $(b,c)$-path,
to transform a partial HC into a $(\leq 2)$-factor, satisfying the second bullet of the claim.

\begin{figure}
\begin{center}
 \begin{tikzpicture}
    \tikzstyle{every node}=[draw,circle,fill=white,minimum size=6pt,
                            inner sep=0pt]


	\node (1) at (0,0) [label=left:$v_1$]{};
 	\node (2) at (2.8,0) [label={[label distance=0.1 cm]above:$b$}]{};
 	\node (3) at (4.4,0) [label={[label distance=0.1 cm]below:$v_j$}] {};
 	\node (4) at (7,0) [label={[label distance=0.1 cm]above:$v_\ell$}] {};
 	\node (5) at (8.5,0) [label={[label distance=0.1 cm]above:$c$}] {};
 	\node (6) at (10.7,0) [label={[label distance=0.1 cm]below:$v_i$}] {};
 	\node (7) at (12.3,0) [label={[label distance=0.1 cm]above:$a$}] {};
 	\node (8) at (15.1,0)[label=right:$v_t$] {};
 	
    \draw[line width = 0.3mm] plot [line width = 0.3mm,smooth,tension=0.6] coordinates{(1) (0.7,-0.1) (1.4,0) (2.1,0.1) (2)};
    \draw[line width = 0.4mm,red] (2) to node[midway,above,draw=none,text=black] {$e_2'$}(3);
    \draw[line width = 0.3mm] plot [smooth,tension=0.8] coordinates{(3) (5.3,-0.06) (6.3,0.1) (4)};
    \draw[line width = 0.4mm,red] (4) to node[midway,above,draw=none,text=black] {$e_3'$}(5);
    \draw[line width = 0.3mm] plot [smooth,tension=0.9] coordinates{(5) (9.2,0.12) (10,-0.1) (6)};
    \draw[line width = 0.4mm,red] (6) to node[midway,above,draw=none,text=black] {$e_1'$}(7);
    \draw[line width = 0.3mm] plot [smooth,tension=0.7] coordinates{(7) (12.8,-0.13) (13.6,0) (14.5,0.1) (8)};


    \draw[line width = 0.3mm,blue] (1) to [out=80,in=100] node[midway,above,draw=none,text=black] {$e_2$} (3);
 	\draw[line width = 0.3mm,blue] (6) to [out=80,in=100] node[midway,above,draw=none,text=black] {$e_1$} (8);
 	\draw[line width = 0.3mm,blue] (2) to [out=-80,in=-100] node[midway,above,draw=none,text=black] {$f$} (5);
	\draw[line width = 0.3mm,blue] (4) to [out=-80,in=-100] node[midway,above,draw=none,text=black] {$e_3$} (7); 	
 	
 	\node (101) at (0,0) {};
 	\node (102) at (2.8,0) {};
 	\node (103) at (4.4,0) {};
 	\node (104) at (7,0) {};
 	\node (105) at (8.5,0) {};
 	\node (106) at (10.7,0) {};
 	\node (107) at (12.3,0) {};
 	\node (108) at (15.1,0) {};

\end{tikzpicture}

\caption{An illustration of the second bullet of Claim \ref{cla:rotator}: Closing the partial HC into $(\leq 2)$-factor. The edges $e_1, e_2, e_3, f$ are added and the edges $e_1', e_2', e_3'$ are deleted. The edge set of the new cycle is $P\cup\{e_1,e_2,e_3,f\}\setminus \{e'_1, e'_2,e'_3\}$.}
\end{center}
\end{figure}
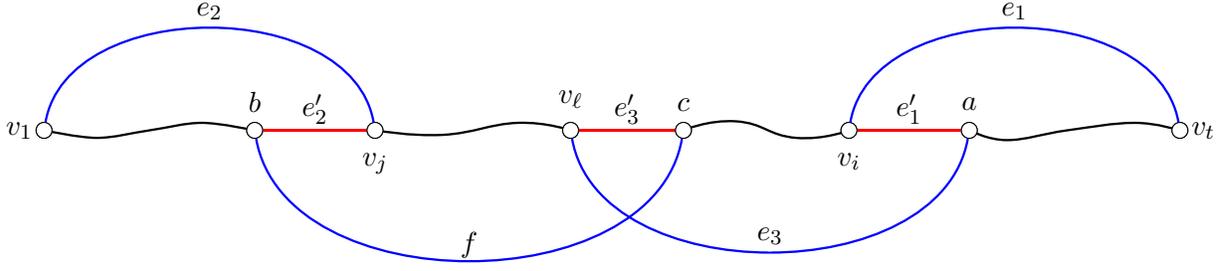

\end{proof}

\begin{observation}
\label{obs:substitution}
For any vertices $x,y\in [n]$ and any $S\subseteq [n]$ of size at most $|S|\leq n^{0.6}$,
there exist vertices $x_1, x_2, y_1, y_2\in [n]\setminus S$
such that $xx_1,yy_1, x_2y_2\in E(G)$ and $x_1x_2, y_1y_2 \in E(F)$.
\end{observation}

\begin{proof}
Since $G$ is $d$-regular with $d \geq \eps c n \geq 2\delta^2n + n^{0.6}$ we can choose
disjoint sets  $A_x\subset N_G(x)$ of size $\delta^2n$
and $A_y\subset N_G(y)$ of size $\delta^2n$ which are also disjoint from $S$ .
Further, we denote $B_x = N_F(A_x)\setminus (S\cup A_x\cup A_y)$
and $B_y = N_F(A_y)\setminus (S\cup A_x\cup A_y)$. Since as we already explained above,
the neighborhood in $F$ of any set of size at least $\delta^2 n$ has size at least $(1/2+\delta)n$,
we conclude that $|B_x|, |B_y|> n/2$.
Since $G$ is regular, every subset $X$ of $G$ satisfies $|N_G(X)| \geq |X|$. Indeed
\[d|X|=\sum_{v \in X} d(v) \leq \sum_{u \in N_G(X)} d(u) =d|N_G(X)|.\]
Therefore there exists an edge $x_2y_2$  in $G$ with $x_2\in B_x$ and $y_2\in B_y$.
Choosing a neighbor $x_1 \in A_x$ of $x_2$ and a neighbor $y_1\in A_y$ of $y_2$ finishes the proof.
\end{proof}

We are now ready to complete the proof of Lemma~\ref{lem:2-factors}. Let $s^*=\sqrt{n\log n}$.
To lower bound the number of $(\leq 2)$-factors with at most $s^*$ components in $G$, we use Corollary~2.8 in \cite{FKS12}.
\begin{proposition}[Corollary~2.8 in \cite{FKS12}]
Let $\tilde{\alpha}> 0$ be a constant and let $n$ be a positive integer. Suppose that:
\begin{itemize}
\item
$\tilde{G}$ is a graph on $n$ vertices, and
\item
$\tilde{G}$ is $\tilde{\alpha}n$-regular.
\end{itemize}
Then the number of $(\leq 2)$-factors of $\tilde{G}$ with at most $s^*$ cycles is $(1 + o(1))^n \lt(\frac{\tilde{\alpha}n}{e}\rt)^n$.
\end{proposition}
Therefore, the number of $(\leq 2)$-factors with at most $s^*$ components in $G$ is at least  $(1+o(1))^n(d/e)^n$.
Fix one such factor and call it $H$. We will transform $H$ into a Hamilton cycle, using edges from
$G$ and $F$. During this process, for the sake of clarity and convenience we will still use $H$ to denote all the new factors and partial HC which we obtain from it.

As long as $H$ is not a Hamilton cycle, we will modify it using  Claim~\ref{cla:rotator} if it is a partial HC, and Claim~\ref{cla:extension}
if $H$ is a $2$-factor. Note that, after at most every two such steps, the number of components in $H$ goes down,
and therefore after at most $2s^*+1$ iterations we end up with a Hamilton cycle $H$.
Since the Hamming distance between the original $(\leq 2)$-factor and the final Hamilton cycle is $O(s^*)$,
the number of possible Hamilton cycles $H$ which we obtain in this process is at least
\[\frac{(1+o(1))^n(d/e)^n}{n^{O(s^*)}}=(1+o(1))^n(d/e)^n.\]

It remains to construct the edge sets $E_F $ and $E_G$, whose role is to keep the resulting graph regular.
We denote ${\cal E}_F = H \cap F$, and note that $|{\cal E}_F| = O(s^*)$.
One by one, for every pair of vertices $x,y\in [n]$ with $xy\in {\cal E}_F$, we apply Observation~\ref{obs:substitution}
to find vertices $x_1, x_2, y_1, y_2$ such that $xx_1,yy_1, x_2y_2\in E(G)$ and $x_1x_2, y_1y_2 \in E(F)$.
We choose it so that none of the vertices $x_1, x_2, y_1, y_2$ is incident to an edge from ${\cal E}_F$ or a previously chosen edge from $E_F$.
We then add the edges $x_1x_2, y_1y_2$ to $E_F$ and the edges $xx_1,yy_1, x_2y_2$ to $E_G$.
Since at most $O(s^*)$ edges from ${\cal E}_F$ and $E_F$ were chosen before, we can always put all endpoints of such edges
into the set $S$ for the further applications of Observation~\ref{obs:substitution}.

We claim that the sets $E_F$ and $E_G$ together with the Hamilton cycle $H$ satisfy the requirements of the lemma.
Indeed, in constructing $F'=F\setminus (H \cup E_F)$ we delete at most $O(s^*)=o(n^{0.6})$ edges from $F$,
and so the number of edges in $F$ between two sets $A$ and $B$ of sizes $|A|\geq \delta^2 n$
and $|B|\geq(1/2-\delta)n$ remains at least $d n^{0.6} - O(s^*)>(d-2)n^{0.6}$.
In addition, note that $G'=(G\cup E_F)\setminus (H\cup E_G)$ is $(d-2)$-regular.
In the transition from $G$ to $G'$, every vertex that isn't incident to an edge of ${\cal E}_F$ loses the two edges of $H$ that pass through it.
On the other hand, for every vertex incident to an edge of ${\cal E}_F$ we remove one edge from $G\cap H$ incident to it,
and then delete an edge from $E_G$ incident to it.
The edges from $E_F$ ensure that this doesn't change the degree of other vertices.
\end{proof}

\vspace{0.15cm}
\noindent
{\it Proof of Theorem \ref{counting}.}\,
Fix an edge partition of the graph $\Gamma$ into graphs $G$, $F$, and $R$ which satisfy the assertions of Lemma~\ref{lem:main}.
Starting with the $d_0$-regular graph $G_0=G$ with $d_0\geq (1-2\eps)r$ and the graph $F_0=F$,
we apply Lemma~\ref{lem:2-factors} repeatedly to remove Hamilton cycles from $G_i \cup F_i$.
After the $i$'th iteration, we will have a $d_i$-regular graph $G_i$ with $d_i=d_0 - 2i$
and a graph $F_i$ satisfying the second bullet of Lemma~\ref{lem:2-factors}. Therefore the process can be continued.
We stop after $t=\frac{d_0-\eps r}{2}$ steps, at which point $G_t$ is an $\eps r$-regular graph.
The bound from Lemma~\ref{lem:2-factors} guarantees that there are at least
$\left((1+o(1))\frac{d_i}{e}\right)^{n}$ choices for a Hamilton cycle in the $i$'th step.
Therefore, this procedure results in at least
\[
\prod_{i=0}^{t-1} \left((1+o(1))\frac{d_i}{e}\right)^{n} =
\left( (1+o(1))^{t} e^{-t} \frac{2^{d_0/2} (d_0/2)!}{2^{\eps r/2} (\eps r/2)!} \right)^{n} \geq
\left( e^{-r/2} \, \frac{(\frac{(1-2\eps)r}{2})!}{(\frac{\eps r}{2})!} \right)^{n} \geq r^{(1-4 \eps)\frac{rn}{2}}
\]
possible ordered tuples of $t$ edge-disjoint Hamilton cycles.

Let us fix one such collection of $t$-edge disjoint Hamilton cycles and
let $R'$ denote the union of $R$ and the edges of $G\cup F$ that do not belong to any of these cycles.
The graph $R'$ is $\alpha n$-regular ($\alpha=3\eps c$), since it is obtained from $\Gamma$ by deleting $t$ edge-disjoint Hamilton cycles.
The graph $R'$ is also a robust $(\nu, \tau)$-expander, since as we mentioned before,
this property is monotone and $R$ satisfies it.
Since $n\geq n_0$, it follows from Theorem~\ref{thm:suff_conditions} that $R'$ admits (at least one) Hamiltonian decomposition.
Together with the above Hamilton cycles it gives us the desired Hamiltonian decomposition of $\Gamma$.

Clearly the same decompositions of $\Gamma$ can appear in our process in at most $(\frac{n-1}{2})!<n^n$ different orderings.
Therefore, we have that $H(\Gamma)\geq r^{(1-4 \eps)\frac{rn}{2}}/n^n \geq r^{(1-5\eps)\frac{rn}{2}}$. \hfill$\Box$

\section{Concluding remarks and open questions}
\label{sec:concl}
\begin{itemize}
\item
A substantial gap remains between the upper bound of Proposition~\ref{thm:upper_bound} and the main result that we prove. We tend to believe that the upper bound of the proposition is close to the truth, i.e, $H(\Gamma)=\left((1+o(1))\frac{r}{e^2}\right)^{rn/2}$
for an $n$-vertex $r$-regular graph $\Gamma$ with $r=cn, c>1/2$. To prove such an estimate using our methods, one needs a Hamiltonian decomposition result for a sparse analogue of robust expanders.
\item
When the degree $r$ of the graph $\Gamma$ is odd, the result of K\"{u}hn and Osthus also implies that the edges of $\Gamma$ can be decomposed into $(r-1)/2$ Hamilton cycles and a perfect matching. Our proof can be easily adapted to show that the number of such decompositions is $r^{(1+o(1))rn/2}$.
\item
It would be interesting to estimate the number of Hamiltonian decompositions of other families of graphs.
In particular, Theorem 1.1.3. in~\cite{CLKOT15+} states that for even $r\geq \lt\lfloor\frac{n}{2}\rt\rfloor$,
every $r$-regular graph on $n$ vertices has a Hamiltonian decomposition.
Can one strengthen our theorem by giving an accurate estimate of the number of such decompositions for all $r$, starting with $\lt\lfloor\frac{n}{2}\rt\rfloor$?
\item
Finally, our main result can be further generalized to subgraphs of quasirandom graphs. A graph $G$ on $n$ vertices  is called
$(\alpha,\beta)$-regular if its minimum degree is at least $\alpha n-1$ and
\[\left|\frac{e_G(S,T)}{|S||T|}-\alpha\right| \leq \beta \]
for every pair of disjoint sets $S, T$ of size at least $\beta n$. It is easy to see that a complete graph on $n$ vertices
is $(1,\beta)$-regular for any $\beta>0$. Another example of such a graph is an $\alpha n$-regular graph on $n$ vertices whose nontrivial eigenvalues in absolute value are all $o(n)$ see, e.g., the survey \cite{KS} for more details.
We can obtain the following result on the number of Hamiltonian decompositions of regular subgraphs of $(\alpha,\beta)$-regular graphs.
\begin{theorem}
\label{quasirandom}
For every constants $ c, \alpha$ such that $c>\alpha/2$ there exists $\beta>0$ such that the following holds. Let $\Gamma$ be an $r$-regular subgraph of an $(\alpha,\beta)$-regular graph $G$ such that $r=cn$ is even. Then the number of Hamiltonian decompositions of $\Gamma$ satisfies $H(\Gamma)=r^{(1+o(1))rn/2}$.
\end{theorem}
When $G$ is a complete graph $K_n$ this implies Theorem \ref{thm:counting_walecki}.

To prove this theorem, using our Theorem \ref{counting}, one needs only to show that
there are at least $\gamma n^2$ edges (for some very small but fixed $\gamma$) between any two subsets $A, |A|=\delta^2 n$ and $B, |B|= (1/2-\delta)n$ of $V(\Gamma)$. Choose $\delta \ll (c-\alpha/2)$ and $\beta\ll \delta^2$. If $\Gamma$ contains two subsets $A, B$ violating the above condition, then by the regularity of $\Gamma$ the number of its edges between $A$ and $V(\Gamma)-B$ is very close to $cn|A|$. On the other hand the graph $G$, containing $\Gamma$, has roughly $\alpha(n-|B|)|A|=
\alpha(1/2+\delta)|A|$ edges between these sets (which is less), a contradiction. We omit further calculation.

\end{itemize}

\vspace{0.15cm}
\noindent
{\bf Acknowledgment.}\,
We would like to thank an anonymous referee for the thorough review of the
paper and for pointing an error in the proof of Lemma \ref*{lem:main}.
Part of this work was done when the third author visited the Freie University Berlin.
He would like to thank the Humboldt Foundation for its generous support during this visit and the
Freie University for its hospitality and stimulating research environment.

\section*{Appendix} \label{app:regular_subgraph}

\begin{proof}[Proof of Lemma \ref{lem:regular_subgraph}]
The proof uses the same ideas as in \cite{KuOs13}, Lemma 5.2 and also \cite{FK}, Section 3.1.

Let $D^*$ be a random orientation of $G^*$, where each edge chooses an orientation uniformly at random and independently of the other edges. We claim that with high probability, $D^*$ satisfies the following.
\begin{itemize}
\item For every vertex $v$ we have the following concentration on the indegrees and outdegrees:
\begin{align*}
c_0 n/2 - n^{2/3} \leq \deg_{in}(v) \leq c_0 n/2 + n^{2/3}, \\
c_0 n/2 - n^{2/3} \leq \deg_{out}(v) \leq c_0 n/2 + n^{2/3}
\end{align*}
\item For every two subsets $A,B \subset V$ such that $|A|\geq c_0 n/3$ and $|B|\geq n/2$ there are at least $\gamma_0 n^2/3$ edges from $A$ to $B$, and at least $\gamma_0 n^2/3$ edges from $B$ to $A$.
\end{itemize}
This follows from a straightforward application of Chernoff's inequality and the union bound.

Recall that $d = \lt\lceil\left(\frac{c_0-\eps_0}{2}\right) n\rt\rceil$.
We construct a subdigraph of $D^*$ whose indegrees and outdegrees are all equal to $d$.
By forgetting the orientation of this subgraph we obtain the desired subgraph of $G^*$.

Consider the following flow network $H$.
The vertex set is $\{s,t\}\cup X \cup Y$,
where $X$ and $Y$ are both copies of $V(G^*)$. We add an edge $(x,y)$ of capacity $1$ for every edge $(x,y)$ in $D^*$.
Furthermore, we connect $s$ to all vertices in $X$ and also connect all vertices in $Y$ to $t$ using edges of capacity $d$.
By the above discussion and the definition of $d$, note that all vertices in $X$ have outdegree
and indegree at least $c_0 n/2 - n^{2/3} >d$ and at most $c_0 n/2 + n^{2/3} $.
Our aim is to show that there is a flow of value $dn$ in $H$.
Clearly, the edges used in such a flow will correspond to a subdigraph of $D^*$ whose indegrees and outdegrees are all $d$, as desired.

By the Max-Flow-Min-Cut Theorem, it is sufficient to show that the capacity of every cut is at least $dn$.
Let $C \subset V(H)$ such that $s \in C$ and $t \notin C$,
and let $S = C \cap X$ and $T = C \cap Y$.
The capacity of the cut defined by $C$ is $d(n-|S|)+e(S,Y \setminus T) + d|T|$,
which is at least $dn$ if and only if $e(S,Y\setminus T) \geq d|S| - d|T|$.
To show that this holds for any $S$ and $T$, we must consider several cases.
Note that we can assume that $|S|>|T|$, since otherwise $d|S|-d|T|\leq 0$ and we are done.

\begin{enumerate}
\item $|S|<d$: In this case, we have
\begin{align*}
e(S,Y \setminus T) = e(S,Y) - e(S,T) \geq (c_0 n/2 - n^{2/3})|S| - |S||T| \geq d |S| - d|T|.
\end{align*}
\item
$|S|-|T|\geq \frac{4 n^{2/3}}{\eps_0}$:
Here, we use the fact that $d$ is substantially smaller than the indegrees and outdegrees of the vertices.
\begin{align*}
 e(S,Y \setminus T) &\geq e(S,Y)-e(X,T) \geq (c_0 n/2 - n^{2/3})|S| - (c_0 n/2 + n^{2/3})|T| \\
& = d|S|-d|T| + \left(\frac{\eps_0 n}{2}\right)(|S|-|T|) - 2n^{2/3}|T| \geq d|S|-d|T|.
\end{align*}
\item $|S| \leq (1-c_0/3) n$: If $|T| \geq n/2$, since $|X \setminus S| \geq c_0 n/3$, by the second property of
$D^*$ we have $e(X\setminus S,T)\geq \frac{\gamma_0 n^2}{3}$.
Otherwise, $|T|<n/2$.
Since after the above cases, we can assume that $|T|\geq |S| - \frac{4 n^{2/3}}{\eps_0}$
and $|S| \geq d$,
we have that $|X \setminus S|\geq (1-o(1))n/2$ and $T \geq c_0 n/3$. Thus, again by the second property of
$D^*$, we have $e(X\setminus S,T)\geq (1-o(1))\frac{\gamma_0 n^2}{3}$.
Indeed, by adding $o(n)$ vertices to $X \setminus S$ we obtain a set of size at least $n/2$,
and this can change the number of edges to $T$ by at most $o(n^2)$.
Hence, in both cases we have $e(X\setminus S,T)\geq (1-o(1))\frac{\gamma_0 n^2}{3}$, and therefore
\begin{align*}
e(S,Y \setminus T) & = e(S,Y) - e(X,T) + e(X \setminus S,T) \\
&  \geq (c_0 n/2 - n^{2/3})|S| - (c_0 n/2 + n^{2/3})|T| + (1-o(1))\frac{\gamma_0 n^2}{3} \geq d|S|-d|T|.
\end{align*}
\item $ S>(1-c_0/3)n$: Here, by the above we assume that $|T| \geq |S|-\frac{4 n^{2/3}}{\eps_0}$.
\begin{align*}
 e(S,Y \setminus T) & = e(X,Y \setminus T) - e(X\setminus S,Y \setminus T) \geq d|Y \setminus T| - |X \setminus S||Y \setminus T| \\
&\geq  d |Y \setminus T| - |X \setminus S|\left(\frac{c_0 n}{3} + \frac{4 n^{2/3}}{\eps_0}\right) \geq d|Y \setminus T|-d|X \setminus S| = d|S|-d|T|.
\end{align*}
\end{enumerate}

\end{proof}

\end{document}